\documentclass[12pt, reqno]{amsart}
\usepackage[english]{babel}   

\usepackage[OT2,T1]{fontenc}
\DeclareSymbolFont{cyrletters}{OT2}{wncyr}{m}{n}
\DeclareMathSymbol{\Sha}{\mathalpha}{cyrletters}{"58}

\usepackage{setspace}     
\usepackage{afterpage} 	
\usepackage{graphicx}	
\usepackage{enumitem} 	

\usepackage{amsmath} 	
\usepackage{amssymb}	
\usepackage{mathrsfs} 	
\usepackage{amsfonts} 	
\usepackage{latexsym} 	
\usepackage[latin1]{inputenc} 
\usepackage{bm}

\usepackage{amsthm} 	
\usepackage{stackrel}       
\usepackage{amscd} 	

\usepackage{tabularx}	
\usepackage{longtable}	

\usepackage{verbatim} 
\usepackage{blindtext} 
\usepackage{lipsum}

\usepackage{subfig}

\allowdisplaybreaks

\def\t{\tau}

\newcommand{\N}{\mathbb{N}} 
\newcommand{\NN}{\mathbb{N}} 
\newcommand{\Z}{\mathbb{Z}} 
\newcommand{\Q}{\mathbb{Q}}

\newcommand{\bz}{\textbf{z}}

\renewcommand{\pmod}[1]{\  \,  \left(  \mathrm{mod} \,  #1 \right)}

\def\sech{\mathop{\rm{sech}}}

\theoremstyle{plain}
\newtheorem{Thm}{Theorem}[section]
\newtheorem*{Thm*}{Theorem}
\newtheorem{corollary}[Thm]{Corollary}
\newtheorem{Conj}[Thm]{Conjecture}
\newtheorem{Lemma}[Thm]{Lemma}

\newtheorem*{Conj*}{Conjecture}

\theoremstyle{definition}

\newtheorem*{defn*}{Definition}
\theoremstyle{remark}

\newtheorem*{rmk*}{Remark}
\newtheorem*{remarks*}{Remarks}

\numberwithin{equation}{section}

\newtheorem{ind}[]{{\rm\it Indice}}

\renewcommand{\epsilon}{\varepsilon}

\usepackage{multicol}

\usepackage{tikz}
\usetikzlibrary{arrows}

\usepackage{cite}

\title[Cranks, Stanton-type conjectures, and unimodality]{Infinite families of crank functions, Stanton-type conjectures, and unimodality}
\author[Bringmann]{Kathrin Bringmann}
\author[Gomez]{Kevin Gomez}
\author[Rolen]{Larry Rolen}
\author[Tripp]{Zack Tripp}
\address{University of Cologne, Department of Mathematics and Computer Science, Weyertal 86-90, 50931 Cologne, Germany}
\email{kbringma@math.uni-koeln.de}
\address{Department of Mathematics, Vanderbilt University, Nashville, TN 37240}
\email{kevin.j.gomez@vanderbilt.edu}
\email{larry.rolen@vanderbilt.edu}
\email[Corresponding author]{zachary.d.tripp@vanderbilt.edu}

\usepackage{listings}

\begin{document}
\numberwithin{equation}{section}

\maketitle

\begin{abstract}
Dyson's rank function and the Andrews--Garvan crank function famously give combinatorial witnesses for Ramanujan's partition function congruences modulo $5$, $7$, and $11$. While these functions can be used to show that the corresponding sets of partitions split into $5$, $7$, or $11$ equally sized sets, one may ask how to make the resulting bijections between partitions organized by rank or crank combinatorially explicit. Stanton recently made conjectures which aim to uncover a deeper combinatorial structure along these lines, where it turns out that minor modifications of the rank and crank are required. 
Here, we prove two of these conjectures.  We also provide abstract criteria for quotients of polynomials by certain cyclotomic polynomials to have non-negative coefficients based on unimodality and symmetry. Furthermore, we extend Stanton's conjecture to an infinite family of cranks. This suggests further applications to other combinatorial objects. We also discuss numerical evidence for our conjectures, connections with other analytic conjectures such as the distribution of partition ranks.
\end{abstract}

\section{Introduction and statement of results}

\subsection{Partition congruences and invariants}\label{Subsection: partition invariants}
Let $p(n)$ be the integer partition function, which counts the number of non-increasing sequences of positive integers that sum to $n\in\N_0$. Ramanujan discovered the following three congruences for the partition function $p(n)$
\begin{align*}
	p(5n+4) \equiv 0 \pmod{5},\quad p(7n+5) \equiv 0 \pmod{7},\quad p(11n+6) \equiv 0 \pmod{11},
\end{align*}
and gave $q$-series proofs of the first two congruences \cite{Ramanujan1920}. Based on an unpublished manuscript of Ramanujan, Hardy found proofs of all three \cite{Hardy1921}. Additionally, Ramanujan conjectured that these were the only congruences of the form $p(\ell n + \beta) \equiv 0 \pmod{\ell}$ for a prime $\ell$, which was later proved by Ahlgren and Boylan \cite{AhlgrenBoylan2003}. Since this time, there has been significant study of other congruences for the partition function \cite{Atkin1968, Ono2000, Ahlgren2000, Weaver2001, Johansson2012}.

In order to give a combinatorial explanation for Ramanujan's congruences, Dyson \cite{Dyson1944} defined the \textit{rank} of a partition to be its largest part minus its number of parts and conjectured that partitions of $5n+4$ (resp. $7n+5$) can be split into $5$ (resp. $7$) sets of equal size by considering the rank modulo $5$ (resp. $7$). This equidistribution of the rank modulo $5$ and $7$ was later proved by Atkin and Swinnerton-Dyer \cite{AtkinSwinnertonDyer1954}. Dyson also conjectured the existence of a partition statistic he called the crank of a partition that was equistributed modulo $11$. Andrews and Garvan \cite{AndrewsGarvan1988} found such a statistic which is equidistributed modulo $5$, $7$, and $11$. For a partition $\lambda$, we let $\ell(\lambda)$ be the largest part of $\lambda$, $\omega(\lambda)$ be the number of $1$'s in $\lambda$, and $\mu(\lambda)$ be the number of parts of $\lambda$ larger than $\omega(\lambda)$. The {\it crank} of $\lambda$ is then defined as
\begin{equation*}
	\mathrm{crank}(\lambda) := 
	\begin{cases}
		\ell(\lambda) & \text{if } \omega(\lambda) = 0,\\
		\mu(\lambda) - \omega(\lambda) & \text{if } \omega(\lambda) > 0.
	\end{cases}
\end{equation*}
In addition to giving a combinatorial explanation for Ramanujan's congruences, the rank and the crank give interesting examples in the theory of modular and mock modular forms. 
The first author and Ono \cite{BringmannOno2010} showed that the rank generating function \eqref{Eqn: rank gen fcn} is essentially a mock Jacobi form, while the crank generating function \eqref{Eqn: crank product} is essentially a meromorphic Jacobi form. Specializing the rank (resp. crank) generating function in the elliptic variable $z$ to a torsion point gives a mock modular (resp. modular) form.

\subsection{Cranks for colored partitions}

Since the discovery of Ramanujan's congruences, many papers have studied similar congruences for other partition related functions \cite{AlanaziMunagiSellers2016, Andrews1984, Mestrige2020}. One such example is the {\it $k$-colored partition} function $p_k(n)$ defined as the number of partitions of $n$ into $k$-colors or by the generating function 
\begin{equation*}
	\sum\limits_{n=0}^\infty p_k(n)q^n := \left( \sum\limits_{n=0}^\infty p(n)q^n \right)^k = \prod\limits_{n=1}^\infty \frac{1}{(1-q^n)^k}.
\end{equation*}
Boylan \cite{Boylan2004} and Dawsey and Wagner \cite{DawseyWagner2017} have proved a number of congruences for $p_k(n)$ using the theory of CM forms and have given partial progress towards classifying all such congruences. Boylan classified all congruences of the form $p_k(\ell n + \beta) \equiv 0 \pmod{\ell}$ for a prime $\ell$ and $k \le 47$ odd and found that all but three such congruences were explained by the following result or by other well-known families of congruences. 
\begin{Thm}[\hspace*{-.15cm}\cite{Boylan2004, DawseyWagner2017}]\label{Thm: k-colored congruences}
	Let $k + h = \ell t$ for a prime $\ell$ and positive integers $h$ and $t$, and let $\delta_{k,\ell}\in\Z$ be such that $24\delta_{k,\ell} \equiv k \pmod{\ell}$. Then we have the Ramanujan-type congruence
	\begin{equation*}
		p_k(\ell n + \delta_{k,\ell})\equiv 0 \pmod{\ell}
	\end{equation*}
	if any of the following hold:
	\begin{enumerate}[leftmargin=*,label=\textnormal{(\arabic*)}]
		\item We have $h \in \{4,8,14\}$ and $\ell \equiv 2 \pmod{3}$. 
		\item We have $h \in \{6,10\}$ and $\ell \equiv 3\pmod{4}$. 
		\item We have $h = 26$ and $\ell \equiv 11 \pmod{12}$. 
	\end{enumerate}
\end{Thm}

Analogous to the rank and crank for $p(n)$, various statistics for $k$-colored partitions have been given. Hammond and Lewis \cite{HammondLewis2004} defined the birank for $2$-colored partitions to explain congruences modulo $5$, while Andrews \cite{Andrews2008} gave a combinatorial interpretation of a certain $2$-color congruence using the ordinary crank. Garvan \cite{Garvan2009} was later able to provide extensions of both of these results in order to explain a certain infinite family of congruences for $k$-colored partitions. More recently, Wagner, the third, and the fourth author \cite{RolenTrippWagner2020} have found two infinite families of cranks that together explain most known congruences for $k$-colored partitions. The generating function for these cranks are defined by certain products of the ordinary crank generating function. The shape of the crank (see \eqref{Eqn: C_k}) is defined in such a way in order to utilize the theory of theta blocks set forth by Gritsenko, Skoruppa, and Zagier \cite{GritsenkoSkoruppaZagier2019}. In \cite{RolenTrippWagner2020}, cranks of this form were multiplied by $1 = \frac{\theta_R}{\theta_R}$ for a given theta block $\theta_R$ depending on the congruence in order to apply a set of sum-to-product identities known as the Macdonald identities to the numerator and denominator separately. This allowed the authors to prove equidistribution in an infinite family of cases. 

\subsection{Stanton's Conjectures}\label{Subsection: Stanton's Conjectures}

While the rank and crank distribute partitions into congruence classes of equal size in order to explain Ramanujan's congruences, there is no known direct map between these equinumerous classes. The search for such a map led to a conjecture of Stanton. In order to state this conjecture, we need to modify the Laurent polynomials $\text{rank}_n(\zeta)$ and $\text{crank}_n(\zeta)$, which are defined in \eqref{Eqn: rank gen fcn} and \eqref{Eqn: crank product}, respectively. 

\begin{defn*}
	For $n\in\N_0$, the \textit{modified rank} and \textit{modified crank} are defined by
	\begin{align}
		\text{rank}_{\ell,n}^*(\zeta) &:= \text{rank}_{\ell n+\beta}(\zeta) + \zeta^{\ell n+\beta-2} - \zeta^{\ell n+\beta-1} + \zeta^{2-\ell n-\beta} - \zeta^{1-\ell n -\beta},\label{Eqn: modified rank def}\\
		\text{crank}_{\ell,n}^*(\zeta) &:= \text{crank}_{\ell n+\beta}(\zeta) + \zeta^{\ell n+\beta-\ell} - \zeta^{\ell n +\beta} + \zeta^{\ell-\ell n-\beta} - \zeta^{-\ell n-\beta},\label{Eqn: modified crank def}
	\end{align}
	where $\beta:=\ell-\frac{\ell^2-1}{24}$. 
\end{defn*}

\begin{rmk*}
Note that these modifications only change the definition of rank and crank for the partitions $n$ and $1+\ldots+1$. For instance, in the case of rank, this assigns the partition $n$ the value $n-2$, although the classical rank assigns it the value $n-1$. 
\end{rmk*}

We see in Lemma~\ref{Lem:equidistribution} below that the explanation of Ramanujan's congruences for $p(n)$ by ranks and cranks is equivalent to the divisibility of rank and crank polynomials by cyclotomic polynomials. Stanton found that the quotients of these rank and crank polynomials by cyclotomic polynomials do not have positive coefficients. The modifications are designed to fix positivity, with the eventual goal of uncovering new combinatorial structure of what these positive coefficients count. Such an interpretation would hopefully yield a map between the congruence classes for the rank and crank.  As we see in Lemma~\ref{Lem: unimodality} below, this positivity is related to unimodality of coefficients. Stanton's modifications essentially fix this unimodality and maintain divisibility by cyclotomic polynomials.

 We are now able to state Stanton's conjecture, which was given in his unpublished notes. Here and throughout the paper, $\Phi_\ell(\zeta):= 1 + \zeta + \ldots + \zeta^{\ell -1}$ denotes the {\it $\ell$-th cyclotomic polynomial} and $\zeta_\ell:=e^{\frac{2\pi i}{\ell}}$.

\begin{Conj}[Stanton]\label{Conj: Stanton} Let $n \in \NN_0$.
	\begin{enumerate}[leftmargin=*,label=\textnormal{(\arabic*)}]
		\item The following are Laurent polynomials with non-negative coefficients:
		\begin{equation*}
			\frac{\mathrm{rank}_{5,n}^*(\zeta)}{\Phi_5(\zeta)} \quad \text{and} \quad \frac{\mathrm{rank}_{7,n}^*(\zeta)}{\Phi_7(\zeta)}.
		\end{equation*}
		\item The following is a Laurent polynomial with positive coefficients:
		\begin{equation*}
			\frac{\mathrm{crank}_{5n+4}(\zeta)}{\Phi_5(\zeta^2)}.
		\end{equation*}
		\item The following are Laurent polynomials with non-negative coefficients: 
		\begin{equation*}
			\frac{\mathrm{crank}_{5,n}^*(\zeta)}{\Phi_5(\zeta)}, \quad \frac{\mathrm{crank}_{7,n}^*(\zeta)}{\Phi_7(\zeta)}, \text{and} \quad \frac{\mathrm{crank}_{11,n}^*(\zeta)}{\Phi_{11}(\zeta)}.
		\end{equation*}
	\end{enumerate}
\end{Conj}
\begin{rmk*}
	In Conjecture~\ref{Conj: Stanton} (2), Stanton used the term ``positive'' instead of ``non-negative''. However, given that the authors found that all examples of the given Laurent polynomial in this case have $0$'s between positive coefficients, we assume that ``Laurent polynomial with positive coefficients'' was meant to be the same as ``Laurent polynomial with non-negative coefficients''.
\end{rmk*}

Stanton also stated a related conjecture of Garvan for the $5$-core crank \cite{GarvanKimStanton1990} that cannot be proven by the methods we give here due to the fact that the $5$-core crank does not appear to be unimodal.  A Laurent polynomial $f(\zeta) = \sum_{m = -M}^N a_m \zeta^m$ is said to be \textit{unimodal} if there exists $k \in\Z$ such that $a_m \le a_{m+1}$ for $m \le k$ and $a_m \ge a_{m+1}$ for $m >k$. Our key to proving Conjecture~\ref{Conj: Stanton} is to use the fact that the modified rank and crank given in Conjecture~\ref{Conj: Stanton} are unimodal.  

\subsection{Results}

\noindent It turns out that the modification Stanton gives for the crank along with known inequalities for the crank are sufficient to prove the following. 

\begin{Thm}\label{Thm: crank conjectures}
	Parts \textnormal{(2)} and \textnormal{(3)} of Conjecture~\ref{Conj: Stanton} are true.
\end{Thm}

\noindent The analogous inequalities for rank are not known, and the authors are unaware of a reference in the literature for such a conjecture, which we give here.

\begin{Conj}\label{Conj: rank unimodality}
	We have $N(m,n) \ge N(m+1, n)$ for $0 \le m \le n-2$ and $n \ge 39$. 
\end{Conj}

Unfortunately, Conjecture~\ref{Conj: rank unimodality} appears to be out of reach with current methods, though there is strong computational evidence and partial progress towards it. The claim is known to be true for $n$ sufficiently large, and for fixed $m$, this can be made explicit. Additionally, Dousse and Mertens \cite{DousseMertens2015} used methods of Dousse and the first author \cite{BringmannDousse2016} to show that for $|m|\le\frac{\sqrt{n}\log(n)}{\pi\sqrt{6}}$, we have as $n\rightarrow\infty$,
\[
N(m,n) = \frac{\gamma}{4} \sech\hspace{-.075cm}^2\left(\frac{\gamma m}{2}\right) p(n) \left(1+O\left(\gamma^{\frac 12} |m|^{\frac 13}\right)\right),
\]
where $\gamma:=\frac{\pi}{\sqrt{6n}}$. Note that $\sech$ is decreasing, so this gives the claim asymptotically. If we assume Conjecture \ref{Conj: rank unimodality}, we are able to prove another part of Conjecture~\ref{Conj: Stanton}.

\begin{Thm}\label{Thm: rank conjecture}
	Conjecture~\ref{Conj: rank unimodality} implies part \textnormal{(1)} of Conjecture~\ref{Conj: Stanton}.
\end{Thm}

It is natural to ask whether Stanton's conjectures are part of a broader phenomenon. Searching for an extension of them may also shed light on their combinatorial interpretation. Recently, in \cite{RolenTrippWagner2020}, Wagner and two of the authors gave a procedure for generating infinite families of crank-type functions which ``explain'' most known congruences for the family of $k$-colored partitions. Thus, it is natural to ask whether a deeper phenomenon like Stanton's conjecture also holds in these cases. Numerically, the authors found that these functions do not typically satisfy Stanton-type conjectures. Moreover, the authors were unable to find simple modifications like Stanton found for rank and crank which ``fixed'' positivity of the quotients by cyclotomic polynomials. However, the method for producing such functions in \cite{RolenTrippWagner2020} is flexible, and the crank-type functions produced are not unique. 

This paper suggests new families of crank-type functions $\mathcal{A}_k(z;\tau)$ and $\mathcal{B}_k(z;\tau)$ (see \eqref{AkBkDefn} for the definition), produced using the same machinery. These still explain most congruences of the colored partitions. The proof is sketched in Section~\ref{SectionStantonColored}.  These invariants appear to satisfy Stanton-type conjectures without any modifications, which suggests that they may be more natural to consider than the original crank-type functions in \cite{RolenTrippWagner2020}, and that Stanton's conjecture appears to be a very general phenomenon deserving an explanation. 

Along these lines, extensive numerical evidence suggests the following.

\begin{Conj}\label{Conj: Unimodality of A and B}
	For all $n \geq 15$, (resp. $24$) and all $k\geq7$, $[q^n]\mathcal{A}_k(z;\tau)$ (resp. $[q^n]\mathcal{B}_k(z;\tau)$) are unimodal Laurent polynomials.
\end{Conj}

Similar to Conjecture~\ref{Conj: rank unimodality}, proving this seems to require delicate analytic techniques that are currently out of reach. Assuming this conjecture, we have the following. 

\begin{Thm}\label{Thm: Stanton for colored partitions}
	Assume Conjecture~\ref{Conj: Unimodality of A and B} is true.
	\begin{enumerate}[leftmargin=*,label=\textnormal{(\arabic*)}]
		\item If $\ell n + \delta_{k,\ell} \ge 15$ and $p_k(\ell n + \delta_{k,\ell}) \equiv 0 \pmod{\ell}$ is a Ramanujan-type congruence coming from Theorem~\ref{Thm: k-colored congruences} with $h \notin \{14, 26\}$ if $k$ is odd and $h \neq 26$ if $k$ is even, then 
		\begin{equation*}
			\frac{\left[q^{\ell n + \delta_{k,\ell}}\right]\mathcal{A}_k(z;\tau)}{\Phi_\ell(\zeta)}
		\end{equation*}
		is a Laurent polynomial with non-negative coefficients.
		\item If $\ell n + \delta_{k,\ell} \ge 24$ and $p_k(\ell n + \delta_{k,\ell}) \equiv 0 \pmod{\ell}$ is a Ramanujan-type congruence coming from Theorem~\ref{Thm: k-colored congruences} with $h \notin \{4, 8, 10, 26\}$ and $k\geq 7$ is odd, then 
		\begin{equation*}
			\frac{\left[q^{\ell n + \delta_{k,\ell}}\right]\mathcal{B}_k(z;\tau)}{\Phi_\ell(\zeta)}
		\end{equation*}
		is a Laurent polynomial with non-negative coefficients.
	\end{enumerate}
\end{Thm}

\begin{rmk*}
	Throughout this paper, we are studying the $n$-th Fourier coefficient (in $\t$) for a fixed $n$ of objects with modularity properties. Traditionally, it has been more common to study the $n$-th Fourier coefficient (in $z$) for a fixed $m$ instead; see e.g. \cite{DabholkarMurthyZagier2012}.
\end{rmk*}

The paper is organized as follows. In Section~\ref{Section: prelim}, we provide a lemma illustrating the connection between divisibility and equidistribution, along with known results about crank equalities and inequalities. We then provide proofs in Section~\ref{Section: proofs} of a lemma, our main theorem, and its corollaries. Finally, Section~\ref{Section: examples} provides computational evidence related to Conjecture~\ref{Conj: Unimodality of A and B} and directions for future research.

\section*{Acknowledgements}

The first author has received funding from the European Research Council (ERC) under the European Union's Horizon 2020 research and innovation programme (grant agreement No. 101001179). 
The third author is grateful for support from a grant from the Simons Foundation (853830, LR), support from a Dean's Faculty Fellowship from Vanderbilt University, and to the Max Planck Institute for Mathematics in Bonn for its hospitality and financial support.
 The authors thank David Chan and Ken Ono for useful discussions related to the topic of the paper. Moreover we thank the referees for carefully reading our paper and making helpful comments. On behalf of all authors, the corresponding author states that there is no conflict of interest and that there is no associated data for this manuscript.

\section{Preliminaries}\label{Section: prelim}

We begin by formally defining terms used in the introduction. We let $N(r,t;n)$ (resp. $M(r,t;n)$) be the number of partitions of $n$ with rank (resp. crank) congruent to $r \pmod{t}$. The equidistribution of the rank modulo $5$ and $7$ is equivalent to 
\begin{align*}
	&N(0,5;5n+4) = N(1,5;5n+4) = \dots = N(4,5;5n+4),\\
	&N(0,7;7n+5) = N(1,7;7n+5) = \dots = N(6,7;7n+5).
\end{align*}
Equidistribution of the crank modulo $5$, $7$, and $11$ may be written similarly in terms of $M(r,t;n)$.  Additionally, we let $N(m,n)$ (resp. $M(m,n)$) be the number of partitions of $n$ with rank $m$ (resp. crank $m$). Letting $\zeta := e^{2\pi i z}$ and $q := e^{2\pi i \tau}$, \cite{AtkinSwinnertonDyer1954} showed that we have the two-parameter generating function
\begin{align}\nonumber
	\mathcal{R}(z;\tau) := &\sum_{\substack{m\in\Z\\n\ge0}}N(m,n)\zeta^m q^n =: \sum_{n=0}^\infty \text{rank}_n(\zeta) q^n\\
	\label{Eqn: rank gen fcn}
	= &\sum_{n=0}^\infty \frac{q^{n^2}}{\prod_{k=1}^n (1-\zeta q^k)(1-\zeta^{-1}q^k)}.
\end{align}
Andrews and Garvan \cite{AndrewsGarvan1988} proved that aside from the anomalous case of $M(m,n)$ if $n=1$ (where the correct values are $M(0,1):=1$ and $M(m,1):=0$ for $m\neq0$) the crank generating function is given by
\begin{equation}\label{Eqn: crank product}
	\mathcal{C}(z;\tau) := \sum_{\substack{m\in\Z\\n\ge0}} M(m,n)\zeta^m q^n =: \sum_{n=0}^\infty \text{crank}_n(\zeta)q^n = \prod_{n=1}^\infty \frac{1-q^n}{(1-\zeta q^n)(1-\zeta^{-1}q^n)}.
\end{equation}
The cranks that the authors in \cite{RolenTrippWagner2020} used are of the form
\begin{equation}\label{Eqn: C_k}
	\mathcal{C}_k\left(a_1, a_2, \dots, a_{\frac{k+\delta_{2\nmid k}}{2}};z;\tau\right) := \mathcal{C}(0;\tau)^{\frac{k-\delta_{2\nmid k}}{2}} \prod_{j=1}^{\frac{k+\delta_{2\nmid k}}{2}} \mathcal{C}(a_j z; \tau),
\end{equation}
where $a_j \in \Z$ for $j = 1, \dots, \frac{k+\delta_{2\nmid k}}{2}$ and where $\delta_S:=1$ if a statement $S$ is true and 0 otherwise. Note that the notation differs slightly from that of \cite{RolenTrippWagner2020} and we choose it since it is more convenient for our purposes.
Recalling \eqref{Eqn: C_k}, we define (in the case of $\mathcal B_k$, for $k\geq7$), 
\begin{equation}\label{AkBkDefn}
\begin{aligned}
	\mathcal{A}_k(z;\tau) &:= \mathcal{C}_k\left(\frac{k+\delta_{2\nmid k}}{2} + 1, \frac{k+\delta_{2\nmid k}}{2}, \dots, 3, 2; z; \tau\right),\\
	\mathcal{B}_k(z;\tau) &:= \mathcal{C}_k\left(\frac{k+\delta_{2\nmid k}}{2} + 2, \frac{k+\delta_{2\nmid k}}{2} + 1, \dots, 6, 5, 3, 2; z; \tau\right).
\end{aligned}
\end{equation}

We begin with a lemma illustrating how Conjecture~\ref{Conj: Stanton} is a statement related to the equidistribution of the rank and crank. This is implicit in the existing literature, but we provide a proof here for the convenience of the reader. By $\Phi_\ell(\zeta)|f(\zeta)$ for $f(\zeta) \in \Q[\zeta^{-1}, \zeta]$, we mean that $f(\zeta) = g(\zeta)\Phi_\ell(\zeta)$ for $g(\zeta) \in \Q[\zeta^{-1}, \zeta]$, i.e., that the quotient $\frac{f(\zeta)}{\Phi_\ell(\zeta)}$ is a Laurent polynomial as well. Define
\[
\widehat{f}_{r,\ell} := \sum_{j \equiv r \pmod{\ell}} \left[\zeta^j\right]f(\zeta).
\]

\begin{Lemma}\label{Lem:equidistribution}
	Let $f(\zeta)$ be a Laurent polynomial in $\Q[\zeta^{-1}, \zeta]$ and $\ell$ a prime. Then $\Phi_\ell(\zeta)\mid f(\zeta)$ in $\Q[\zeta^{-1}, \zeta]$ if and only if for $r\in\{0,\dots,\ell-2\}$
	\begin{equation*}
		\widehat{f}_{r,\ell} = \widehat{f}_{\ell - 1,\ell}.
	\end{equation*} 
\end{Lemma}

\begin{rmk*}
	Letting $f(\zeta)$ be $\text{rank}_{\ell n + \beta}(\zeta)$ or $\text{crank}_{\ell n + \beta}(\zeta)$, we see that divisibility by $\Phi_\ell(\zeta)$ is equivalent to equidistribution modulo $\ell$. We use this result frequently.
\end{rmk*}

\begin{proof}[Proof of Lemma \ref{Lem:equidistribution}]
 Multiplying $f(\zeta)$ by a sufficiently large power of $\zeta$ and using the fact that $\gcd(\zeta, \Phi_\ell(\zeta)) = 1$, we may assume that $f(\zeta) \in \Q[\zeta]$. Since $\Phi_\ell(\zeta)$ is irreducible over $\Q[\zeta]$, it is a standard fact from algebra that $\Phi_\ell(\zeta)\mid f(\zeta)$ is equivalent to $f(\zeta_\ell) = 0$. Writing $f(\zeta) = \sum_{j = 0}^n a_j \zeta^j$, we see that
	\begin{equation*}\label{Eqn: root of unity evaluation}
		f(\zeta_\ell) = \sum_{j=0}^n a_j \zeta_\ell^j = \sum_{r=0}^{\ell -1}  \sum_{\substack{0 \leq j \leq n \\ j \equiv r\pmod{\ell}}} a_j  \zeta_\ell^r = \sum_{r = 0}^{\ell - 1} \widehat{f}_{r,\ell} \zeta_\ell^r = \sum_{r=0}^{\ell - 2} \left(\widehat{f}_{r,\ell} - \widehat{f}_{\ell - 1,\ell}\right) \zeta_\ell^r,
	\end{equation*}
	where for the last equality we use the fact that $\Phi_\ell (\zeta_\ell) = 0$. Since $1, \zeta_{\ell}, \dots, \zeta_{\ell}^{\ell -2}$ is a basis for $\Q[\zeta]$ over $\Q$, the claim follows. 
\end{proof}

Lemma \ref{Lem:equidistribution} can be generalized to equidistribution modulo prime powers by requiring divisibility by multiple cyclotomic polynomials, but we omit the details since we are only interested in equidistribution modulo primes in this paper. However, we utilize a modified version of the above lemma when we require divisibility by $\Phi_5(\zeta^2)$ for the proof of part (2) of Conjecture~\ref{Conj: Stanton}. The following is used in conjunction with Theorem~\ref{Thm: crank mod 10}, which in particular satisfies the conditions of this lemma.

\begin{Lemma}\label{Lem: modified equidistribution}
	Let $f(\zeta)$ be a Laurent polynomial and $\ell$ an odd prime. Then $\Phi_\ell(-\zeta)\mid f(\zeta)$ in $\Q[\zeta^{-1}, \zeta]$ if and only if for $r\in\{0,1,\dots,\ell-2\}$,
	\begin{equation*}\label{Eqn: mod 2 ell divisibility}
		(-1)^r \left(\widehat{f}_{r,2\ell} - \widehat{f}_{r+\ell,2\ell}\right) = \widehat{f}_{\ell - 1,2\ell} - \widehat{f}_{2\ell - 1,2\ell}.
	\end{equation*}
\end{Lemma}

\begin{proof}
	 As in the proof of Lemma \ref{Lem:equidistribution}, we may assume that $f(\zeta) \in \Q[\zeta]$. As $\Phi_\ell(-\zeta)$ is also an irreducible polynomial in $\Q[\zeta]$, $\Phi_\ell(-\zeta)\mid f(\zeta)$ is equivalent to $f(-\zeta_\ell) = 0$ since $-\zeta_\ell$ is a root of $\Phi_\ell (-\zeta)$.  Writing $f(\zeta) = \sum_{j=0}^n a_j \zeta^j$, we see that
	\begin{align}\label{Eqn: negative root of unity evaluation}
		f(-\zeta_\ell) &= \sum_{j=0}^n a_j (-\zeta_\ell)^j = \sum_{r=0}^{\ell -1} \left( \sum_{\substack{0 \leq j \leq n \\ j \equiv r\pmod{2\ell}}} (-1)^{j} a_j + \sum_{\substack{0 \leq j \leq n \\ j \equiv \ell + r \pmod{2\ell}}} (-1)^j a_j \right) \zeta_\ell^r \nonumber \\
		&= \sum_{r = 0}^{\ell - 1} (-1)^r \left(\widehat{f}_{r,2\ell} - \widehat{f}_{r+\ell,2\ell}\right) \zeta_\ell^r \nonumber\\
		&= \sum_{r=0}^{\ell - 2} \left((-1)^r\left( \widehat{f}_{r,2\ell} - \widehat{f}_{r+\ell ,2\ell} \right) - \left(\widehat{f}_{\ell - 1,2\ell} - \widehat{f}_{2\ell - 1,2\ell} \right) \right)\zeta_\ell^r.\nonumber
	\end{align}
	Since $1, \zeta_{\ell}, \ldots , \zeta_{\ell}^{\ell - 2}$ is a basis for $\Q[\zeta_\ell]$ over $\Q $, we conclude the claim.
\end{proof}

We now review known results concerning the crank that allow us to prove Theorem~\ref{Thm: crank conjectures}. First, in order to prove (2) of Conjecture~\ref{Conj: Stanton}, we need additional relationships among cranks modulo $10$ given by Garvan \cite{Garvan1990}. 

\begin{Thm}[(1.17) and (1.18) of \cite{Garvan1990}]\label{Thm: crank mod 10}
	For $n\in\N_0$, $0 \le k \le 4$, $j\in\{0,1\}$
	\[
		M(2k+j,10;5n+4) = \frac 15 M(j,2;5n+4).
	\]
\end{Thm}

Additionally, in order to utilize Theorem~\ref{Lem: unimodality}, we need the following result of Ji and Zang \cite{JiZang2018} related to crank unimodality.

\begin{Thm}[Theorem 1.7 of \cite{JiZang2018}]\label{Thm: crank unimodality}
	For $n \ge 44$ and $1 \le m \le n-1$, 
	\begin{equation*}
		M(m-1,n) \ge M(m,n).
	\end{equation*}
\end{Thm}
Unfortunately, $\text{crank}_n(\zeta)$ is not actually unimodal for $n \ge 44$ because $M(n-1, n) = 0$ and $M(n,n) = 1$ for $n \ge 2$. In order to see this, we have the following result, which is needed below in the paper.
\begin{Lemma}\label{Lem: constant crank}
	For  fixed $k\in\N$, the sequence $M(n-k,n)$ is constant for $n \ge 2k$.
\end{Lemma}
\begin{rmk*}
As alluded to above, Lemma \ref{Lem: constant crank} proves that $\text{crank}_n(\zeta)$ is not unimodal. We see in the proof of Theorem~\ref{Thm: crank conjectures} below that $\text{crank}_{\ell,n}^*(\zeta)$ is unimodal, thus illustrating the need for the modified crank function that Stanton provides.
\end{rmk*}

\begin{proof}[Proof of Lemma \ref{Lem: constant crank}]
	We utilize the following summation formula for $m \in \N$ \cite[Theorem 7.19]{Garvan1988}
	\begin{equation}\label{Eqn: crank1}
		\sum\limits_{n=0}^\infty M(m,n)q^n = \frac{1}{\prod_{k=0}^\infty (1-q^k)} \sum\limits_{n=1}^\infty (-1)^{n-1} q^{\frac{n(n-1)}{2} + mn}.
	\end{equation}
	By replacing $m$ by $m+1$ and dividing both sides by $q$, we find that
	\begin{equation*}
		\sum\limits_{n=-1}^\infty M(m+1,n+1)q^n = \frac{1}{\prod_{k=0}^\infty (1-q^k)} \sum\limits_{n=1}^\infty (-1)^{n-1} q^{\frac{n(n-1)}{2} + (m+1)n-1}.
	\end{equation*}
	However, since $m+1 > 0$, we may conclude that $M(m+1,0) = 0$, so
	\begin{equation}\label{Eqn: crank2}
		\sum\limits_{n=0}^\infty M(m+1,n+1)q^n = \frac{1}{\prod_{k=0}^\infty (1-q^k)} \sum\limits_{n=1}^\infty (-1)^{n-1} q^{\frac{n(n-1)}{2} + (m+1)n-1}.
	\end{equation}
	Subtracting \eqref{Eqn: crank1} from \eqref{Eqn: crank2} yields
	\begin{multline}\label{Eqn: diff of cranks}
		\sum\limits_{n=0}^\infty (M(m+1, n+1) - M(m,n)) q^n\\
		=\frac{1}{\prod_{k=0}^\infty (1-q^k)} \sum\limits_{n=1}^\infty (-1)^{n-1} \left( q^{\frac{n(n-1)}{2} + (m+1)n-1} - q^{\frac{n(n-1)}{2} + mn}\right).
	\end{multline}
	We now claim that the $j$-th Fourier coefficient vanishes for $j \le 2m$. Note that the term $n = 1$ vanishes in \eqref{Eqn: diff of cranks}. On the other hand for $n \ge 2$, 
	\begin{equation*}
		\frac{n(n-1)}{2} + (m+1)n-1 \ge \frac{n(n-1)}{2} + mn \ge 2m+1, 
	\end{equation*}
	so the smallest power of $q$ in \eqref{Eqn: diff of cranks} is at least $2m+1$. Comparing coefficients on both sides of the equality in \eqref{Eqn: diff of cranks}, this tells us that $M(m+1,n+1) = M(m,n)$ for $n \le 2m$. Replacing $m$ by $n-k$ yields the result.
\end{proof}

\section{Proof of the main results}\label{Section: proofs}

\subsection{A general result}

We first need the following lemma, where a Laurent polynomial $f(\zeta)$ is called \textit{symmetric} if $f(\zeta^{-1})=f(\zeta) $. 

\begin{Lemma}\label{Lem: unimodality}
	Let $f(\zeta)$ be a symmetric unimodal Laurent polynomial that is divisible by $\Phi_\ell(\zeta)$ for an odd prime $\ell$. Then the coefficients of the Laurent polynomial $\frac{f(\zeta)}{\Phi_\ell(\zeta)}$ are non-negative.
\end{Lemma}

\begin{rmk*}
	Note that if $f(\zeta)$ is strictly unimodal, then the coefficients of $\frac{f(\zeta)}{\Phi_\ell (\zeta)}$ are positive.
\end{rmk*}

\begin{proof}[Proof of Lemma \ref{Lem: unimodality}]
Write 
$$
(1-\zeta)f(\zeta) =: \sum_m a_m \zeta^m.
$$
 By the symmetry of $f(\zeta)$, we have
\begin{equation*}
	\sum_m a_m \zeta^{-m} 
	= \left(1-\zeta^{-1}\right)f\left(\zeta^{-1}\right) = -\zeta^{-1}(1-\zeta)f(\zeta) = -\sum_m a_m \zeta^{m-1}.
\end{equation*}
Comparing coefficients, we conclude that 
\begin{equation}\label{amm}
	a_m = -a_{-m+1}.
\end{equation}
 We next show that 
\begin{equation}\label{Eqn: a_m positivity}
	a_m \ge 0 \text{ for } m \le 0 \text{ and } a_m \le 0 \text{ for } m \ge 1.
\end{equation}
By \eqref{amm}, we only need to prove $a_m \le 0$ for $m\ge 1$. If $f(\zeta) = \sum_m c_m \zeta^m$, then comparing coefficients for $(1-\zeta)f(\zeta)$ yields  $a_m = c_m - c_{m-1}$. By the unimodality of $f(\zeta)$, we conclude that $c_m \le c_{m-1}$ for $m \ge 1$ yielding $a_m\le0$ as desired.

Now, we consider $\frac{f(\zeta)}{\Phi_\ell(\zeta)}$. Note that $(1-\zeta)\Phi_\ell(\zeta) = 1-\zeta^\ell$. Thus for $|\zeta| < 1$,
\begin{align*}
	\sum_m b_m \zeta^m := \frac{f(\zeta)}{\Phi_\ell(\zeta)}  = \frac{(1-\zeta)f(\zeta)}{1-\zeta^\ell} = \left( 1 + \zeta^{\ell} + \zeta^{2\ell} + \dots \right) \sum_m a_m \zeta^m.
\end{align*}
Comparing coefficients, we find that $b_m = \sum_{j=0}^\infty a_{m-\ell j}$. By \eqref{Eqn: a_m positivity}, the non-negativity of $b_m$ is then clear for $m \le 0$. For a fixed $m \ge 1$, let $k$ be sufficiently large so that $b_{m + \ell k} = 0$. Such a $k$ exists because $\Phi_\ell(\zeta)\mid f(\zeta)$ implies that $b_j = 0$ for sufficiently large $j$. Then 
\begin{equation*}
	b_m = b_{m} - b_{m+\ell k} = \sum_{j=0}^\infty a_{m-\ell j} -  \sum_{j=0}^\infty a_{m+\ell k-\ell j} = -\sum_{j=0}^{k-1} a_{m+\ell k - \ell j} \ge 0
\end{equation*}
by \eqref{Eqn: a_m positivity}. Thus, we conclude $b_m \ge 0$ for all $m$.
\end{proof}

In order to prove Conjecture~\ref{Conj: Stanton} (2), we require a slight modification of Lemma~\ref{Lem: unimodality}. The proof is similar to the proof of Lemma~\ref{Lem: unimodality}.

\begin{Lemma}\label{Lem: modified unimodality}
	Let $f(\zeta)$ be a symmetric Laurent polynomial that is divisible by $\Phi_\ell(\zeta^2)$ for a prime $\ell$ and such that $[\zeta^{m-1}]f(\zeta) \ge [\zeta^{m+1}]f(\zeta)$ for $m\in\N$. Then $\frac{\zeta^{\ell - 1}f(\zeta)}{\Phi_\ell(\zeta^2)}$ is a symmetric Laurent polynomial with non-negative coefficients.
\end{Lemma}

\begin{proof}
	We write
	\begin{equation*}
		(\zeta^{-1}-\zeta)	f(\zeta) =: \sum\limits_{m = -N}^{N} d_m \zeta^m \quad \text{and} \quad \frac{\zeta^{\ell - 1}f(\zeta)}{\Phi_\ell(\zeta^2)} =: \sum_{k=-N+\ell}^{N-\ell} e_k \zeta^k.
	\end{equation*}
	By the symmetry of $f(\zeta)$, we have
	\begin{equation*}
		\sum_m d_m \zeta^{-m} = \left(\zeta-\zeta^{-1}\right)f\left(\zeta^{-1}\right) = -\left(\zeta^{-1}-\zeta\right)f(\zeta) = -\sum_m d_m \zeta^{m}.
	\end{equation*}
	Comparing coefficients, we conclude that 
	\begin{align}\label{Eqn: symm}
		d_m = -d_{-m}.
	\end{align}
	Note that in particular $d_0 = -d_0$, so $d_0 = 0$. We next show that 
	\begin{equation}\label{Eqn: a_m second}
		d_m \ge 0 \text{ for } m \le 0 \text{ and } d_m \le 0 \text{ for } m \ge 1.
	\end{equation}
	By \eqref{Eqn: symm}, we only need to prove that $d_m \le 0$ for $m\ge 1$. Writing $f(\zeta) = \sum_m c_m \zeta^m$, then comparing coefficients for $(\zeta^{-1}-\zeta)f(\zeta)$ yields the equality 
	$$d_m = c_{m+1} - c_{m-1}.$$
	Additionally, by assumption of the lemma we have that
	$$ c_{m-1} = \left[\zeta^{m-1}\right]f(\zeta) \ge \left[\zeta^{m+1}\right] f(\zeta) = c_{m+1},$$
	from which $d_m\le 0$ follows.
	
	We next note that $\frac{\zeta^{\ell - 1}f(\zeta)}{\Phi_\ell(\zeta^2)}$ is symmetric since $f(\zeta)$ and $\zeta^{\frac{1-\ell}{2}} \Phi_\ell(\zeta)$ are.

	Next we find a formula for $e_k$ in terms of $d_m$. For $|\zeta| < 1$, we have
	\begin{align*}
		\sum_k e_k \zeta^k = \frac{\zeta^\ell\left(\zeta_{-1}-\zeta\right)f(\zeta)}{1-\zeta^{2\ell}} = \zeta^{\ell}\left( 1 + \zeta^{2\ell} + \zeta^{4\ell} + \dots \right) \sum_m d_m \zeta^m.
	\end{align*}
	Comparing coefficients, we obtain that 
	\begin{equation*}
			e_k = \sum_{j=0}^\infty d_{k-\ell (2j+1)}.
	\end{equation*}
	By \eqref{Eqn: a_m second}, the non-negativity of $e_k$ is then clear for $k \le \ell$ since $k - \ell(2j+1) \le 0$ for such values. For a fixed $k \ge \ell + 1$, let $r$ be sufficiently large so that $e_{k + 2\ell r} = 0$. Such an $r$ exists because $\Phi_\ell(\zeta^2)\mid f(\zeta)$ implies that $e_n = 0$ for sufficiently large $n$. Then
	\begin{equation*}
		e_k = e_{k} - e_{k+2\ell r} = \sum_{j=0}^\infty d_{k-\ell (2j+1)} -  \sum_{j=0}^\infty d_{k+2\ell r-\ell(2j+1)} = -\sum_{j=0}^{k-1} d_{k+2\ell(r-j) - \ell} \ge 0
	\end{equation*}
	by \eqref{Eqn: a_m second} since all of the indices in the final sum are negative. Thus, we conclude $e_k \ge 0$ for all $m$.
	\end{proof}

\subsection{Proof of Theorem~\ref{Thm: crank conjectures}}

\begin{proof}
	We begin with the proof of Conjecture~\ref{Conj: Stanton} (3). The polynomials under consideration are of the form \eqref{Eqn: modified crank def}  for $\ell \in \{5,7,11\}$. We check the conditions of Lemma~\ref{Lem: unimodality} for these polynomials. The symmetry of \eqref{Eqn: modified crank def} follows from the fact that $M(m,n) = M(-m,n)$ for all $m,n \in \Z$, which can be seen from the symmetry of $\mathcal{C}(z;\tau)$ under $z\mapsto-z$.  Now, we check that \eqref{Eqn: modified crank def} is divisible by $\Phi_\ell(\zeta)$. By Lemma~\ref{Lem:equidistribution}, the divisibility of $\text{crank}_{\ell n + \beta}(\zeta)$ by $\Phi_\ell(\zeta)$ is equivalent to
	\begin{equation*}
		M(0, \ell; \ell n + \beta) = M(1, \ell; \ell n + \beta) = \ldots = M(\ell -1, \ell; \ell n + \beta).
	\end{equation*}
	For $\ell \in \{5,7,11\}$, this is simply the statement of the well-known equidistribution of the crank (see \cite[Vector-crank Theorem]{AndrewsGarvan1988}). On the other hand, 
	\begin{align*}
		\zeta^{\ell n + \beta - \ell} - \zeta^{\ell n + \beta} + \zeta^{\ell -\ell n - \beta} - \zeta^{-\ell n - \beta} &= \left(\zeta^{\ell n + \beta - \ell}-\zeta^{-\ell n - \beta}\right)\left(1-\zeta^{\ell}\right) \\
		&= \left(\zeta^{\ell n + \beta - \ell}-\zeta^{-\ell n - \beta}\right)(1-\zeta)\Phi_\ell(\zeta)
	\end{align*}
	proves the divisibility of the remaining part of \eqref{Eqn: modified crank def}.
	Now, we show the unimodality of $\text{crank}_{\ell n + \beta}^*(\zeta)$ for $\ell n + \beta \ge 44$. By symmetry, it suffices to show that the coefficients are decreasing for non-negative powers of $\zeta$. First, note that for $m \ge \ell n + \beta + 1$,
	\begin{align*}
		[\zeta^m]\text{crank}_{\ell,n}^*(\zeta) &= M(m, \ell n +\beta) = 0.
	\end{align*}
	Additionally, if $m = \ell n + \beta$, then
	\begin{align*}
		\left[\zeta^{\ell n + \beta}\right]\text{crank}_{\ell,n}^*(\zeta) &= M(\ell n + \beta, \ell n + \beta) - 1 = 0.
	\end{align*}
	Now, for $0 \le m \le \ell n + \beta - 1$,  we note that $[\zeta^m]\text{crank}_{\ell,n}^*(\zeta) = M(m, \ell n + \beta)$
	except if $m = \ell n + \beta - \ell$. Thus, by Theorem~\ref{Thm: crank unimodality},
	\begin{equation*}
		\left[\zeta^m\right]\text{crank}_{\ell, n}^*(\zeta) - \left[\zeta^{m+1}\right]\text{crank}_{\ell,n}^*(\zeta) \ge M(m,\ell n+ \beta) - M(m+1,\ell n + \beta) \ge 0
	\end{equation*}
	for $0 \le m \le \ell n + \beta -2$ and $m \neq \ell n + \beta - \ell - 1$. In order to prove the inequality for $m = \ell n + \beta - \ell - 1$, we note for a fixed value $k$, the sequence $\{M(n-k,n)\}_{n=1}^\infty$ is constant for $n \ge 2k$ by Lemma~\ref{Lem: constant crank}, so it suffices to check that
	\begin{align*}
		\left[\zeta^{\ell n + \beta-\ell - 1}\right]&\text{crank}_{\ell,n}^*(\zeta) - \left[\zeta^{\ell n + \beta - \ell}\right]\text{crank}_{\ell,n}^*(\zeta) \\
		&= M(\ell n + \beta - \ell -1,\ell n + \beta) - M(\ell n + \beta-\ell,\ell n + \beta) - 1 \ge 0
	\end{align*}
	for $\ell \in \{ 5,7,11\}$ and $n = 22$, which the authors have checked by computer.  As a result, we have unimodality of $\text{crank}_{\ell,n}^*(\zeta)$ for $\ell n + \beta \ge 44$. For $\ell n + \beta < 44$, the result can be checked manually. Additionally, if $m = \ell n + \beta - 1$, then 
	{\small 
	\begin{equation*}
		\left[\zeta^{\ell n + \beta - 1}\right]\text{crank}_{\ell, n}^*(\zeta) - \left[\zeta^{\ell n + \beta}\right]\text{crank}_{\ell, n}^*(\zeta) = M(\ell n + \beta - 1, \ell n + \beta) - M(\ell n + \beta, \ell n + \beta) + 1 = 0,
	\end{equation*}
	}
	completing the proof of unimodality.

	We now move to the proof of Conjecture~\ref{Conj: Stanton} (2). In order to check the inequality condition of Lemma~\ref{Lem: modified unimodality},  note that $M(m,5n+4) \ge M(m+2,5n+4)$ for $0 \le m \le 5n+1$ and $5n+4 \ge 44$ by Theorem~\ref{Thm: crank unimodality} and can be checked manually for $5n+4 < 44$. As for $m = 5n+2$, it is easy to check that $M(5n+2,5n+4) = M(5n+4,5n+4) = 1$ for $5n+4 \ge 2$, and for $m \ge 5n+3$, the inequality follows from the fact that $M(5n+3, 5n+4) = 0$ and $M(m,5n+4) = 0$ for $m > 5n+4$. This proves that $[\zeta^m]\text{crank}_{5n+4}(\zeta) \ge [\zeta^{m+2}]\text{crank}_{5n+4}(\zeta)$ for $5n+4 \ge 2$. Additionally, the divisibility by $\Phi_5 (\zeta^2)$ follows directly from Lemma~\ref{Lem: modified equidistribution} and Theorem~\ref{Thm: crank mod 10}, so the result follows from Lemma~\ref{Lem: modified unimodality}.
\end{proof}

\subsection{Proof of Theorem~\ref{Thm: rank conjecture}}

\begin{proof}
We again use Lemma~\ref{Lem: unimodality}. The polynomials under consideration are given in \eqref{Eqn: modified rank def} for $\ell \in \{5,7\}$. The symmetry $N(m,n) = N(-m,n)$ can be seen by the invariance of $\mathcal{R}(z;\tau)$ under $z\mapsto-z$. The symmetry for the remaining terms in \eqref{Eqn: modified rank def} is clear. As for unimodality, the assumption of Conjecture~\ref{Conj: rank unimodality} means that it suffices to show that for $m \ge \ell n + \beta-3$,
\[
\left[\zeta^m\right]\text{rank}_{\ell, n}^*\left(\zeta\right) \ge \left[\zeta^{m+1}\right]\text{rank}_{\ell, n}^*\left(\zeta\right).
\]
 From the definition of the rank, it is easy to check that the partition $\ell n + \beta$ of $\ell n + \beta$ is the only partition with rank $\ell n + \beta -1$. Similarly, we may check that there are no partitions of rank $\ell n + \beta-2$ or of rank $m \ge \ell n + \beta$, while the partition $(\ell n + \beta-1,1)$ is the only partition of rank $\ell n + \beta -3$ for $n \ge 1$. Therefore, from the definition \eqref{Eqn: modified rank def}, we see that the unimodality of $\text{rank}_{\ell, n}^*(\zeta)$ holds. Finally, the divisibility of $\text{rank}_{\ell n + \beta}(\zeta)$ by $\Phi_\ell(\zeta)$ for $\ell \in \{5,7\}$ follows from the well-known equidistribution of the rank modulo $5$ and $7$ and from Lemma~\ref{Lem:equidistribution}. Additionally, 
\begin{equation*}
	\zeta_\ell^{\ell n + \beta -2} - \zeta_\ell^{\ell n + \beta -1} + \zeta_\ell^{2-\ell n - \beta} - \zeta_\ell^{1-\ell n - \beta} = \zeta_\ell^{\beta -2} - \zeta_\ell^{\beta - 1} +\zeta_\ell^{2-\beta} - \zeta_\ell^{1-\beta},
\end{equation*}
and by plugging in $5$ and $7$ for $\ell$, we can see that the latter term of \eqref{Eqn: modified rank def} vanishes~under $\zeta_{\ell}$ and hence is divisible by $\Phi_\ell(\zeta)$. Thus, Lemma~\ref{Lem: unimodality} completes the proof of Theorem~\ref{Thm: rank conjecture}. 
\end{proof}

\subsection{Sketch of proof of Theorem~\ref{Thm: Stanton for colored partitions}}\label{SectionStantonColored}

We provide only a sketch of the proof in order to avoid reintroducing the entire framework of \cite{RolenTrippWagner2020} for a simple modification of those results. We leave it to the interested reader to make the necessary changes to the proofs in \cite{RolenTrippWagner2020}. For such a reader, we point out that there are results analogous to Lemmas 4.1 and 4.2 that apply to $\mathcal{A}_k$ and $\mathcal{B}_k$. The first lemma allows one to analyze $\mathcal{A}_k$, which explains almost all of the congruences coming from Theorem~\ref{Thm: k-colored congruences}. Below, we use the notation $g(\zeta;q) \equiv h(\zeta;q)\pmod{\Phi_\ell(\zeta)}$ to mean that $\Phi_\ell(\zeta)$ divides $g(\zeta;q)-h(\zeta;q)$ in the ring $\Q[[q]][\zeta, \zeta^{-1}]$.

\begin{Lemma}\label{denomfactors1}
	Suppose that $\{\ell n+\delta_{k,\ell}\}_{n\in\N_0}$ is an arithmetic progression coming from Theorem~\ref{Thm: k-colored congruences} for $p_k(n)$ with $k + h = \ell t$ and $h \in \{4,6,8,10\}$ if $k$ is odd and $k \in \{4,6,8,10,14\}$ if $k$ is even. Then if $\phi_R(\bz;\tau)$ is the theta block associated to $h$ in  \cite[Table 1]{RolenTrippWagner2020}, then there is a choice of $a, b \in \Z$ such that if $\bz = (az, bz)$, then
	\begin{multline*}
		\phi_R(\bz;\tau) \left(\zeta^{\pm 2} q\right)_\infty \cdot \ldots \cdot \left(\zeta^{\pm \frac{k+\delta_{2\nmid k}}{2}}q\right)_\infty \left(\zeta^{\pm \left(\frac{k+\delta_{2\nmid k}}{2} + 1\right)}q\right)_\infty 
		\\
		\equiv q^{\frac{h}{24}}f(\zeta)(q)_\infty^{\delta_{2\nmid k}}\left(q^{\ell}; q^{\ell}\right)_\infty^t \pmod{\Phi_\ell(\zeta)}
	\end{multline*}
	for some $f(\zeta) \in \Q[\zeta^{\frac 12}, \zeta^{-\frac 12}]$.
\end{Lemma}

Unfortunately, Lemma \ref{denomfactors1} does not appear to hold if $k \equiv -14 \pmod{\ell}$ and $k$ is odd. However, we have the following analogous result that provides a combinatorial description in this case using $\mathcal{B}_k$. 

\begin{Lemma}\label{denomfactors2}
	Suppose that $\{\ell n+\delta_{k,\ell}\}_{n\in\N_0}$ is an arithmetic progression coming from Theorem~\ref{Thm: k-colored congruences} for $p_k(n)$ with $k + h = \ell t$ and $h \in \{6,14\}$ and $k$ odd. Then if $\phi_R(\bz;\tau)$ is the theta block associated to $h$ in \cite[Table 1]{RolenTrippWagner2020}, then there is a choice of $a, b \in \Z$ such that
	\begin{multline*}
		\phi_R(\bz;\tau) \left(\zeta^{\pm 2} q\right)_\infty\left(\zeta^{\pm 3} q\right)_\infty\left(\zeta^{\pm 5} q\right)_\infty \cdot \ldots \cdot \left(\zeta^{\pm \frac{k+3}{2}}q\right)_\infty \left(\zeta^{\pm \frac{k+5}{2}}q\right)_\infty 
		\\\equiv q^{\frac{h}{24}}(q)_\infty f(\zeta) \left(q^{\ell};q^{\ell}\right)_\infty^t \pmod{\Phi_\ell(\zeta)}
	\end{multline*} 
	for some $f(\zeta) \in \Q[\zeta^{\frac 12}, \zeta^{-\frac 12}]$.
\end{Lemma}

One difference between the above lemmas and those in \cite{RolenTrippWagner2020} is that the missing residues in our case may depend on the value of $k$. However, they are still easy to determine and can be filled in by an appropriate choice of $a$ and $b$ as was done in \cite{RolenTrippWagner2020}. Following the proof of \cite[Theorem 1.3]{RolenTrippWagner2020}, we have the following result.

\begin{corollary}\label{Cor: theta blocks}
	If $\{\ell n + \delta_{k, \ell}\}_{n \in \N_0}$ is an arithmetic progression coming from Theorem~\ref{Thm: k-colored congruences} with $k + h = \ell t$ for $h \in \{4,6,8,10,14\}$ if $k$ is even and $h \in \{4,6,8,10\}$ if $k$ is odd, then
	\begin{equation*}
		\Phi_\ell(\zeta) \mid \left[q^{\ell n + \delta_{k, \ell}}\right]\mathcal{A}_k(z;\tau).
	\end{equation*}
	Similarly, if $\{\ell n + \delta_{k, \ell}\}_{n \in \N_0}$ is an arithmetic progression coming from Theorem~\ref{Thm: k-colored congruences} with $k + h = \ell t$, $h \in \{6,14\}$, and $k\geq 7$ odd, then 
	\begin{equation*}
		\Phi_\ell(\zeta) \mid \left[q^{\ell n + \delta_{k, \ell}}\right]\mathcal{B}_k(z;\tau).
	\end{equation*}
\end{corollary}

Using this result, we are able to prove Theorem~\ref{Thm: Stanton for colored partitions}.

\begin{proof}[Proof of Theorem~\ref{Thm: Stanton for colored partitions}]
To illustrate the symmetry of our polynomials, recall that $\text{crank}_n(\zeta)$ is symmetric. Thus, the coefficients of $\mathcal{A}_k(z;\tau)$ and $\mathcal{B}_k(z;\tau)$ are products and sums of symmetric polynomials, so they must be symmetric as well.  Additionally, the unimodality of the coefficients is given as an assumption, while the divisibility of the coefficients is given by Corollary~\ref{Cor: theta blocks}. As a result, the proof is finished by Lemma~\ref{Lem: unimodality}.
\end{proof}

\section{Numerical evidence of conjectures and ideas for further work}\label{Section: examples}

\subsection{Computations}\label{Subsection: computational}

The computational evidence supporting Conjecture~\ref{Conj: Unimodality of A and B} was found initially through an exhaustive search of crank generating functions to find likely eventually unimodal examples. The search space is given by
\begin{equation*}
	\mathcal{S} := \bigcup_{3 \leq k \leq 11} \left\{ \mathcal{C}_k\left(a_1,a_2,\dots,a_{\frac{k+\delta_{2\nmid k}}{2}}; z; \tau\right) : k \geq a_1 > a_2 > \dots > a_{\frac{k+\delta_{2\nmid k}}{2}} > 0\right\}.
\end{equation*}
Note that the space is unrestricted by any consideration of the cranks' ability to explain $k$-colored partition congruences.

For each crank generating function $\mathcal{D} \in \mathcal{S}$, the minimum value of $m$ such that $[q^n]\mathcal{D}$ is unimodal for all $m < n < 75$ was found, with $\mathcal{D}$ being considered a likely eventually unimodal candidate if such an $m$ exists. The entire search was completed in approximately 56 hours on a single-threaded Intel i7-8750H CPU. The search space for larger $k$ increases exponentially since the number of possible cranks is given by the $\lfloor \frac k2 \rfloor$-th central binomial coefficient, with the total computation runtime then being in $O(2^k)$.

The set of all likely eventually unimodal examples was then searched manually for general trends and potential infinite families, with $\mathcal{A}_k$ and $\mathcal{B}_k$ emerging as the families of choice due to their determined unimodality and ability to explain almost all known $k$-colored partition congruences. Conjecture~\ref{Conj: Unimodality of A and B} was then formulated and more extensively verified for all $3 \leq k \leq 20$ and the respective $n \leq 99$; this verification took approximately 4 hours on the same hardware.

Evaluating the set of likely eventually unimodal examples also led to the following general conjecture. See Table~\ref{Table: cranks} for a partial summary of our computations.

{\small

\begin{table}[ht]
	\centering
	\subfloat[$k=3$]{
		\begin{tabular}{|c|c|}
			\hline
			Crank & Unimodal? \\
			\hline
			$\mathcal{C}_3(2, 1; z; \tau)$ & $\forall n > 7$ \\
			$\mathcal{C}_3(3, 1; z; \tau)$ & no \\
			$\mathcal{C}_3(3, 2; z; \tau)$ & $\forall n > 6$ \\
			\hline
	\end{tabular}}
	\qquad 
	\subfloat[$k=4$]{
		\begin{tabular}{|c|c|}
			\hline
			Crank & Unimodal? \\
			\hline
			$\mathcal{C}_4(2, 1; z; \tau)$ & $\forall n > 1$ \\
			$\mathcal{C}_4(3, 1; z; \tau)$ & no \\
			$\mathcal{C}_4(4, 1; z; \tau)$ & no \\
			$\mathcal{C}_4(3, 2; z; \tau)$ & $\forall n > 1$ \\
			$\mathcal{C}_4(4, 2; z; \tau)$ & no \\
			$\mathcal{C}_4(4, 3; z; \tau)$ & $\forall n > 23$ \\
			\hline
	\end{tabular}}
	\qquad
	\subfloat[$k=5$] {
		\begin{tabular}{|c|c|}
			\hline
			Crank & Unimodal? \\
			\hline
			$\mathcal{C}_5(3, 2, 1; z; \tau)$ & $\forall n > 9$ \\
			$\mathcal{C}_5(4, 2, 1; z; \tau)$ & no \\
			$\mathcal{C}_5(5, 2, 1; z; \tau)$ & no \\
			$\mathcal{C}_5(4, 3, 1; z; \tau)$ & $\forall n > 11$ \\
			$\mathcal{C}_5(5, 3, 1; z; \tau)$ & no \\
			$\mathcal{C}_5(5, 4, 1; z; \tau)$ & $\forall n > 9$ \\
			$\mathcal{C}_5(4, 3, 2; z; \tau)$ & $\forall n > 10$ \\
			$\mathcal{C}_5(5, 3, 2; z; \tau)$ & no \\
			$\mathcal{C}_5(5, 4, 2; z; \tau)$ & $\forall n > 13$ \\
			$\mathcal{C}_5(5, 4, 3; z; \tau)$ & $\forall n > 13$ \\
			\hline
	\end{tabular}}
	\qquad
	\subfloat[$k = 6$]{
		\begin{tabular}{|c|c|}
			\hline
			Crank & Unimodal? \\
			\hline
			$\mathcal{C}_6(3, 2, 1; z; \tau)$ & $\forall n > 1$ \\
			$\mathcal{C}_6(4, 2, 1; z; \tau)$ & no \\
			$\mathcal{C}_6(5, 2, 1; z; \tau)$ & no \\
			$\mathcal{C}_6(6, 2, 1; z; \tau)$ & no \\
			$\mathcal{C}_6(4, 3, 1; z; \tau)$ & $\forall n > 5$ \\
			$\mathcal{C}_6(5, 3, 1; z; \tau)$ & no \\
			$\mathcal{C}_6(6, 3, 1; z; \tau)$ & no \\
			$\mathcal{C}_6(5, 4, 1; z; \tau)$ & $\forall n > 11$ \\
			$\mathcal{C}_6(6, 4, 1; z; \tau)$ & no \\
			$\mathcal{C}_6(6, 5, 1; z; \tau)$ & $\forall n > 21$ \\
			$\mathcal{C}_6(4, 3, 2; z; \tau)$ & $\forall n > 14$ \\
			$\mathcal{C}_6(5, 3, 2; z; \tau)$ & no \\
			$\mathcal{C}_6(6, 3, 2; z; \tau)$ & no \\
			$\mathcal{C}_6(5, 4, 2; z; \tau)$ & $\forall n > 19$ \\
			$\mathcal{C}_6(6, 4, 2; z; \tau)$ & no \\
			$\mathcal{C}_6(6, 5, 2; z; \tau)$ & $\forall n > 20$ \\
			$\mathcal{C}_6(5, 4, 3; z; \tau)$ & $\forall n > 7$ \\
			$\mathcal{C}_6(6, 4, 3; z; \tau)$ & no \\
			$\mathcal{C}_6(6, 5, 3; z; \tau)$ & $\forall n > 32$ \\
			$\mathcal{C}_6(6, 5, 4; z; \tau)$ & $\forall n > 19$ \\
			\hline
	\end{tabular}}
	\caption{Cranks for the given value of $k$}
	\label{Table: cranks}
\end{table}

}

\newpage

\begin{Conj} \label{Conj: Unimodality over all cranks}
	Let $\mathcal{D}(z;\tau) := \mathcal{C}_k(a_1,a_2,\dots,a_{\frac{k+\delta_{2\nmid k}}{2}}; z; \tau)$ for some $a_1 > a_2 > \dots > a_{\frac{k+\delta_{2\nmid k}}{2}} > 0$ and $k \geq 3$. Then $\mathcal{D}(z;\tau)$ is eventually unimodal if and only if $a_1 - a_2 = 1$. 
\end{Conj}

\subsection{Questions and ideas for further research}

We conclude with open questions for further study.

\begin{enumerate}[leftmargin=*]
	\item Recalling the motivation for Conjecture~\ref{Conj: Stanton} given in Section~\ref{Subsection: Stanton's Conjectures}, are there combinatorial interpretations of the coefficients in the conjectures?
	\item Similarly, is there a combinatorial explanation of the non-negativity of the coefficients given in Theorem~\ref{Thm: Stanton for colored partitions}?
	\item Use Lemma~\ref{Lem:equidistribution} and Lemma~\ref{Lem: unimodality} (or modifications of them such as ones we have given) to prove non-negativity of coefficients for polynomials related to other families of partition functions, to congruences of the partition function modulo higher powers of primes, or to other combinatorial objects. 
	\item Prove or give partial results towards any of the unimodality conjectures given such as Conjecture~\ref{Conj: rank unimodality} and Conjecture~\ref{Conj: Unimodality of A and B}. 
\end{enumerate}

\end{document}